\documentclass[preprint, authoryear]{elsarticle}

\usepackage{lineno,hyperref}
\usepackage{amsmath, amsthm, amssymb}
\usepackage{mathtools}
\usepackage{bbm}
\usepackage{csquotes}
\usepackage{pgfplots}
\usepackage{pgfplotsthemetol}
\pgfplotsset{compat=1.11}
\usepgfplotslibrary{fillbetween}
\usepgfplotslibrary{groupplots}
\usepackage{tikz}
\usepackage{tikz-cd}
\usetikzlibrary{arrows, positioning}
\usepackage{subcaption}
\usepackage{enumitem}

\newtheorem{theorem}{Theorem}[section]
\newtheorem{definition}[theorem]{Definition}
\newtheorem{lemma}[theorem]{Lemma}
\newtheorem{corollary}[theorem]{Corollary}
\newtheorem{proposition}[theorem]{Proposition}
\newtheorem{remark}[theorem]{Remark}
\newtheorem{example}[theorem]{Example}

\biboptions{authoryear}

\usepackage[margin=3cm]{geometry}

\providecommand{\R}{\ensuremath{\mathbb{R}}}
\providecommand{\C}{\ensuremath{\mathcal{C}}}	
\providecommand{\Cd}[1]{\C_{#1}}				

\providecommand{\rbraces}[1]{\left( #1 \right)} 		
\providecommand{\cbraces}[1]{\left[ #1 \right]}			
\providecommand{\curly}[1]{\left\{ #1 \right\}} 		
\providecommand{\abs}[1]{\left\lvert #1 \right\rvert} 	

\providecommand{\set}[2]{\left\{ #1 \mid #2 \right\}} 
\providecommand{\indFunc}[1]{\mathbbm{1}_{#1}} 
\providecommand{\lebesgue}[3]{\int\limits_{#2} #1 \ \mathrm{d}#3}		
\providecommand{\rInt}[4]{\int\limits_{#2}^{#3} #1 \ #4}				
\providecommand{\cInt}[4]{\rInt{#1}{#2}{#3}{\mathrm{d}#4}}				

\providecommand{\ir}[1]{#1^{SI}}			
\providecommand{\dr}[1]{#1^{SD}}			
\providecommand{\CB}[1]{{C}^{\#}(#1)}		
\providecommand{\ordSum}[2]{\left\langle #1 \, , \, #2 \right\rangle}	

\begin{document}

\begin{frontmatter}

\title{Stochastic monotonicity and the Markov product for copulas\tnoteref{t1,t2}}
\tnotetext[t1]{\textcopyright 2021. This manuscript version is made available under the CC-BY-NC-ND 4.0 license \url{http://creativecommons.org/licenses/by-nc-nd/4.0/}.}
\tnotetext[t2]{Accepted for publication in the Journal of Mathematical Analysis and Applications (\hyperlink{https://doi.org/10.1016/j.jmaa.2021.124942}{10.1016/j.jmaa.2021.125348}).}

\address[tud]{Faculty of Mathematics, TU Dortmund University, Germany}%
\author[tud]{Karl Friedrich Siburg}\corref{correspondingauthor}%
\cortext[correspondingauthor]{Corresponding author.}%
\author[tud]{Christopher Strothmann\fnref{scholarship}}%
\fntext[scholarship]{Supported by the German Academic Scholarship Foundation.}%

\begin{abstract}
Given two random variables $X$ and $Y$, stochastic monotonicity describes a monotone influence of $X$ on $Y$. 
We prove two different characterizations of stochastically monotone $2$-copulas using the isomorphism between $2$-copulas and Markov operators. 
The first approach establishes a one-to-one correspondence between stochastically monotone copulas and monotonicity-preserving Markov operators.
The second approach characterizes stochastically monotone copulas by their monotonicity property with respect to the Markov product. 
Applying the latter result, we identify all idempotent stochastically monotone copulas as ordinal sums of  the independence copula $\Pi$.
\end{abstract}

\end{frontmatter}

%
%
%


\section{Introduction}

Describing the relationship between random variables is at the core of dependence modelling. 
In this work, we will focus on stochastic monotonicity, a concept that captures a monotone influence of one random variable $X$ on another one $Y$. 
More precisely, we call $Y$ \emph{stochastically increasing} in $X$ whenever the corresponding conditional distribution functions are pointwise decreasing, i.e.\ one has
\begin{equation*} \label{eq:introduction_stoch_mon}
	\mathbb{P} (Y \leq y \; | \; X = x_2) \leq \mathbb{P} (Y \leq y \mid X = x_1)
\end{equation*}
for almost all $x_1 \leq x_2$ in the support of $X$; in other words, $[Y \mid X = x_1]$ is stochastically dominated by $[Y \mid X = x_2]$.
Similarly, $Y$ is stochastically decreasing in $X$ if $\mathbb{P} (Y \leq y \mid X = x)$ is increasing in $x$. 

Stochastic monotonicity has become an important tool in applications, e.g.\ to study the long-term behaviour of economic models such as stochastic recursions of the form $X_{n+1} = f(X_n, Z_n)$ a.s. (see \cite{Stokey.1989} or \cite{Foss.2018} for details). 
Concrete data examples which exhibit stochastic monotonicity are given by the connection between expenditures and income of a household or the income mobility from one generation to the next, and can be found in \cite{Lee.2009}.

To investigate the connection between stochastic monotonicity and the Markov product for copulas, it is necessary to assume that $X$ and $Y$ are continuous random variables.
In this case, $X$ and $Y$ have a unique copula $C$ defined on $[0,1]^2$, and $Y$ is stochastically increasing, respectively decreasing, in $X$ if and only if $u \mapsto \partial_1 C(u, v)$ is decreasing, respectively increasing, for almost all $u\in [0,1]$ and all $v \in [0, 1]$.
It is well-known that this is equivalent to the concavity, respectively convexity, of $C$ with regard to the first component; any such copula will be called \emph{stochastically monotone}. The concept of stochastic monotonicity for random variables or copulas has been investigated in the literature under various different terms, including \emph{conditionally increasing property} (\cite{Mueller.2001}), \emph{positive regression dependence} (\cite{Lehmann.1959}), and \emph{concavity in $u$}. 

Related, but much more restrictive, is the notion of componentwise concavity requiring that a copula $C$ is concave in \emph{each} component, when the other one is held fixed; see, e.g. \cite{Alvoni.2007}, \cite{Durante.2009} and \cite{Dolati.2014}.
One obvious drawback of this concept is the loss of a directed influence between random variables, as it suggests a circular interaction: an increase in $X$ leads to an increase in $Y$, which in turn leads to higher values in $X$, and so on. 

The aim of this work is to characterize stochastically monotone copulas and their corresponding Markov operators, and to investigate the connection between stochastic monotonicity and the Markov product of copulas introduced in \cite{Darsow.1992} as 
\begin{equation*}
	(C_1*C_2) (u, v)	= \cInt{\partial_2 C_1(u, t) \partial_1 C_2(t, v)}{0}{1}{t} ~.
\end{equation*}
Central to our study is the result by \cite{Olsen.1996} stating that the set of copulas endowed with the Markov product is isomorphic to the set of Markov operators on $L^1([0, 1])$ equipped with the composition. 

Building upon this isomorphism, we show in Theorem~\ref{thm:characterization} that stochastically monotone copulas are in one-to-one correspondence with monotonicity-preserving Markov operators. As a by-product, this implies that the set of stochastically monotone copulas is closed under the action of the Markov product.

Next, Theorem~\ref{thm:characterization_mp_si} characterizes stochastically increasing copulas by their monotonicity property under the Markov product with respect to the stochastic dominance ordering.
More precisely, we prove that a copula $C$ is stochastically increasing if and only if $D * C \leq C$ holds pointwise for all copulas $D$.  
This is used in Theorem~\ref{thm:sm_copula} to identify all idempotent (i.e., $C*C=C$) stochastically monotone copulas as ordinal sums of the independence copula $\Pi$. 

Finally, we apply that the algebraic property of a copula $C$ being idempotent translates into the stochastic property of its Markov operator $T_C$ being a conditional expectation on $L^\infty([0, 1], \mathcal{B}([0, 1]), \lambda)$. This implies Theorem~\ref{thm:conditional_expectation_char} stating that a conditional expectation is monotonicity-preserving or -reversing if and only if it is pointwise an averaging operator. 

The paper is structured as follows.
Section~\ref{section:notation} introduces the necessary notation and definitions. 
Section~\ref{section:characterization} provides a characterization of stochastically monotone copulas in terms of Markov operators as well as some topological closure properties of stochastically monotone copulas.
Section~\ref{section:markov_product} contains the aforementioned characterization of stochastic monotonicity in terms of the Markov product.
Section~\ref{section:idempotents} uses the previous characterizations to identify the idempotent, stochastically monotone copulas as ordinal sums of $\Pi$.
\section{Notation and preliminaries} \label{section:notation}

In this section, we introduce our notation and collect some preliminary results.

A $2$-copula is a bivariate distribution function on $[0, 1]^2$ with uniform margins.
We refer to the lower Fréchet-Hoeffding-bound by $C^-$, to the product copula by $\Pi$ and to the upper Fréchet-Hoeffding-bound by $C^+$; see \cite{Nelsen.2006} for details.
\cite{Darsow.1992} first introduced a product structure on the set of all bivariate copulas $\Cd{2}$, which constitutes a continuous analogue to the multiplication of doubly stochastic matrices.

\begin{definition}
The Markov product of two $2$-copulas $C_1,C_2$ is the $2$-copula
\begin{equation*}
	(C_1 * C_2) (u, v) := \cInt{\partial_2 C_1(u, t) \partial_1 C_2(t, v)}{0}{1}{t} ~.
\end{equation*}
\end{definition}

Note that the partial derivative of a copula $C$ is only defined almost everywhere. 
The Markov product has been applied, for example, in the study of extremal points of $\Cd{2}$ (see, e.g., \cite{Darsow.1992}) and the treatment of complete dependence. 
Furthermore, $(\Cd{2}, *)$ is closely linked to a class of integral-preserving linear operators equipped with the composition.

\begin{definition} \label{def:markov_operator}
A linear operator $T: L^1([0, 1]) \rightarrow L^1([0, 1])$ is called a Markov operator if
\begin{enumerate}
	\item $T$ is positive, that is $Tf \geq 0$ if $f \geq 0$.
	\item $T$ has the fixed point $\indFunc{[0, 1]}$.\label{def:markov_operator_fixpoint} 
	\item $T$ is integral-preserving, i.e.
	\begin{equation*}
		\cInt{Tf(t)}{0}{1}{t} = \cInt{f(t)}{0}{1}{t} 
	\end{equation*}
	holds for all $f \in L^1([0, 1])$.
\end{enumerate}
\end{definition}

\cite{Olsen.1996} established the existence of an isomorphism between $2$-copulas and Markov operators which translates the Markov product into the composition of the corresponding Markov operators.

\begin{theorem} \label{thm:cop_mo_correspondence}
Let $C$ be a $2$-copula and $T$ be a Markov operator. Then
\begin{equation*}
	C_T(u, v) := \cInt{T\indFunc{[0, v]}(t)}{0}{u}{t} \; \text{ and } \; T_C f (u) := \partial_u \cInt{\partial_v C(u, v) f(v)}{0}{1}{v} 
\end{equation*}
define a $2$-copula and a Markov operator, respectively. 
This correspondence is one-to-one with $T_{C_T} = T$ and $C_{T_C} = C$ for all $2$-copulas $C$ and all Markov operators $T$.
Moreover, for all $2$-copulas $C_1$ and $C_2$, it holds that $$T_{C_1 * C_2} = T_{C_1} \circ T_{C_2}~.$$
\end{theorem}

Markov operators are closely related to conditional expectations (see, \cite{Durante.2015} and \cite{Eisner.2015} for an in-depth treatment of this connection). 

\begin{proposition} \label{prop:idempotent_CE}
For a $2$-copula $C$ and its corresponding Markov operator $T_C$, the following assertions are equivalent
\begin{enumerate}
	\item $C$ is idempotent, i.e. $C * C = C$.
	\item $T_C$ is idempotent, i.e. $T_C \circ T_C = T_C$.
	\item $T_C$ is a conditional expectation restricted to $L^\infty([0, 1], \mathcal{B}([0, 1]), \lambda)$, i.e.
	\begin{equation*}
		T_C f = \mathbb{E} (f  \mid \mathcal{G})
	\end{equation*}
	holds for all $f \in L^1$, where $\mathcal{G} := \set{A \in \mathcal{B}([0, 1])}{T_C \indFunc{A} = \indFunc{A}}$.   

\end{enumerate}
\end{proposition}

Lastly, we will review the concept of ordinal sums, see, for example, \cite{Nelsen.2006}.

\begin{definition}
Let $\rbraces{(a_k, b_k)}_{k \in \mathcal{I}}$ be a countable family of disjoint intervals in $(0, 1)$ and $(C_k)_{k \in \mathcal{I}}$ a family of $2$-copulas.
A $2$-copula $C$ is called an ordinal sum of $(C_k)_{k \in \mathcal{I}}$ with respect to $\rbraces{(a_k, b_k)}_{k \in \mathcal{I}}$ if
\begin{equation*}
	C(u, v)	= 	\begin{cases}
					a_k + (b_k - a_k) C_k \rbraces{\frac{u - a_k}{b_k - a_k}, \frac{v - a_k}{b_k - a_k}}	&\text{ if } (u, v) \in (a_k, b_k)^2 \\
					C^+(u, v)																				&\text{ else}
				\end{cases} ~.
\end{equation*}
We use the short-hand notation $C = \rbraces{\ordSum{(a_k, b_k)}{C_k}}_{k \in \mathcal{I}}$ presented in \cite{Durante.2015}. 
\end{definition}
\section{Stochastic monotonicity for copulas and Markov operators} \label{section:characterization}

\begin{definition}
A $2$-copula $C$ is called stochastically increasing (decreasing) in the $i$-th component if $u_i \mapsto \partial_i C(u_1, u_2)$ is decreasing (increasing) for almost all $u_i \in [0, 1]$. 
\end{definition}
Whenever the meaning is clear, we will drop the specification \enquote{in the $i$-th component}.
If the distinction between $u \mapsto \partial_i C(u, v)$ being increasing or decreasing for all $v \in [0, 1]$ is of no concern, we will call $C$ simply stochastically monotone.
Furthermore, we will state many results only with respect to the first component, from which the result in the other component follows by transposition.   
We denote the set of all in the first component stochastically increasing (decreasing) copulas  by $\ir{\C}$ ($\dr{\C}$).
Note that stochastically \emph{increasing} $2$-copulas have a \emph{decreasing} partial derivative.
This is due to the fact that for real random variables $X$ and $Y$, $X \leq_{st} Y$ if and only if $$F_X(t) \geq F_Y(t)$$ holds for all $t \in \R$. 

We will begin by presenting some well-known examples of stochastically monotone $2$-copulas. 

\begin{example}
An Archimedean $2$-copula
\begin{equation*}
	C (u, v) = \phi (\phi^{[-1]}(u) + \phi^{[-1]}(v))
\end{equation*}
with twice-differentiable additive generator $\phi$ is stochastically increasing in both components if and only if $\log \rbraces{- \phi'}$ is convex (see Proposition~3.3 in \cite{Capra.1993}). 
The independence copula $\Pi$ with generator $\phi (t) = \exp(-t)$ is stochastically increasing. 
\end{example}

\begin{example}
The class of extreme value copulas 
\begin{equation*}
	C (u, v) = \exp \rbraces{ \log (uv) A \rbraces{\frac{\log u}{\log uv}}} ~,
\end{equation*}
where $A: [0, 1] \rightarrow [0, 1]$ is a convex function fulfilling $\max \rbraces{t, 1-t} \leq A(t) \leq 1$ for $t \in [0, 1]$, is stochastically increasing in both components (see, Theorem~1 in \cite{Garralda.2000}). 
This class includes both $\Pi$ and $C^+$ as examples with $A(t) = 1$ and $A(t) = \max \rbraces{t, 1-t}$, respectively.
\end{example}

All of these examples are stochastically increasing in both components, thus not allowing for an only unidirectional positive influence between the random variables. 
Many common construction methods using $2$-copulas as building blocks also preserve the stochastic increasing property, such as convex combinations and ordinal sums of stochastically increasing $2$-copulas (see \cite{Durante.2009}).
A $2$-copula which is stochastically increasing in the first component, but not in the second is given by the following example. 

\begin{example}
A straight-forward calculation shows that the checkerboard-copula 
\begin{equation*}
	C^\#_3 (A)(u, v)	= \sum\limits_{k, \ell = 1}^3 a_{k\ell} \cInt{\indFunc{\cbraces{\frac{k-1}{n}, \frac{k}{n}}}(s)}{0}{u}{s} \cInt{\indFunc{\cbraces{\frac{\ell-1}{n}, \frac{\ell}{n}}}(t)}{0}{v}{t} 	~,	
\end{equation*}
with the doubly stochastic matrix $A = (a_{k\ell})_{k, \ell = 1, 2, 3} \in \R^{3 \times 3}$ 
\begin{equation*}
	A = \begin{pmatrix}
			2/3	&0		&1/3 \\
			1/3	&1/3	&1/3 \\
			0	&2/3	&1/3
		\end{pmatrix} 
\end{equation*}
is stochastically increasing in the first but not the second component. 
A plot of the partial derivatives of $C^\#_3(A)$ is depicted in Figure~\ref{fig:checkerboard_SI_NSI}.
\begin{figure}
\centering
\begin{subfigure}[b]{0.49\textwidth}
	\centering
    \resizebox{\linewidth}{!}{\begin{tikzpicture}[declare function={
    pCB1(\u)=and(\u >= 0, \u < 0.33)*(0.67) + and(\u >= 0.33, \u < 0.67)*(0.33) + and(\u >= 0.67, \u <= 1)*(0);
    pCB2(\u)=and(\u >= 0, \u <=1)*(0.67);
    }]
	\begin{axis}[xmin=0, xmax=1, ymin=0, ymax=1, xtick={0, 0.33, 0.67, 1}, ytick={0, 0.33, 0.67, 1}]
    	\foreach \xStart/\xEnd  in {0/0.329, 0.33/0.67, 0.67/1} {
        	\addplot[domain=\xStart:\xEnd, dashdotted] {pCB1(x)};
    	}
    	\addplot[fill,only marks,mark=*] coordinates{(0.33,0.67)};
    	\addplot[fill=white,only marks,mark=*] coordinates{(0.33,0.33)};
    	\draw[dotted] (axis cs:0.33,0.67) -- (axis cs:0.33,0.33);
    	\addplot[fill,only marks,mark=*] coordinates{(0.67,0.33)};
    	\addplot[fill=white,only marks,mark=*] coordinates{(0.67,0)};
    	\draw[dotted] (axis cs:0.67,0.33) -- (axis cs:0.67,0);
	\end{axis}
\end{tikzpicture}}
    \caption{Plot of $\partial_1 C^\#_3 (A)(\cdot, 1/3)$.}
    \label{figSubA:checkerboard_SI_NSIFirst}
\end{subfigure}
\begin{subfigure}[b]{0.49\textwidth}
	\centering
    \resizebox{\linewidth}{!}{\begin{tikzpicture}[declare function={
    pCB1(\u)=and(\u >= 0, \u < 0.33)*(0.67) + and(\u >= 0.33, \u < 0.67)*(0) + and(\u >= 0.67, \u <= 1)*(0.33);
    pCB2(\u)=and(\u >= 0, \u < 0.33)*(1) + and(\u >= 0.33, \u < 0.67)*(0.33) + and(\u >= 0.67, \u <= 1)*(0.67);
    }]
	\begin{axis}[xmin=0, xmax=1, ymin=0, ymax=1, xtick={0, 0.33, 0.67, 1}, ytick={0, 0.33, 0.67, 1}]
    	\foreach \xStart/\xEnd  in {0/0.329, 0.33/0.67, 0.67/1} {
        	\addplot[domain=\xStart:\xEnd, dashdotted] {pCB1(x)};
    	}
    	\addplot[fill,only marks,mark=*] coordinates{(0.33,0.67)};
    	\addplot[fill=white,only marks,mark=*] coordinates{(0.33,0)};
    	\draw[dotted] (axis cs:0.33,0.67) -- (axis cs:0.33,0);
    	\addplot[fill,only marks,mark=*] coordinates{(0.67,0)};
    	\addplot[fill=white,only marks,mark=*] coordinates{(0.67,0.33)};
    	\draw[dotted] (axis cs:0.67,0.33) -- (axis cs:0.67,0);
%
	\end{axis}
\end{tikzpicture}}
    \caption{Plot of $\partial_2 C^\#_3 (A)(1/3, \cdot)$.}
    \label{figSubB:checkerboard_SI_NSISecond}
\end{subfigure}
\caption{Plots of the partial derivative of the checkerboard copula $\CB{A}$ with respect to the first and second component.}
\label{fig:checkerboard_SI_NSI}
\end{figure}
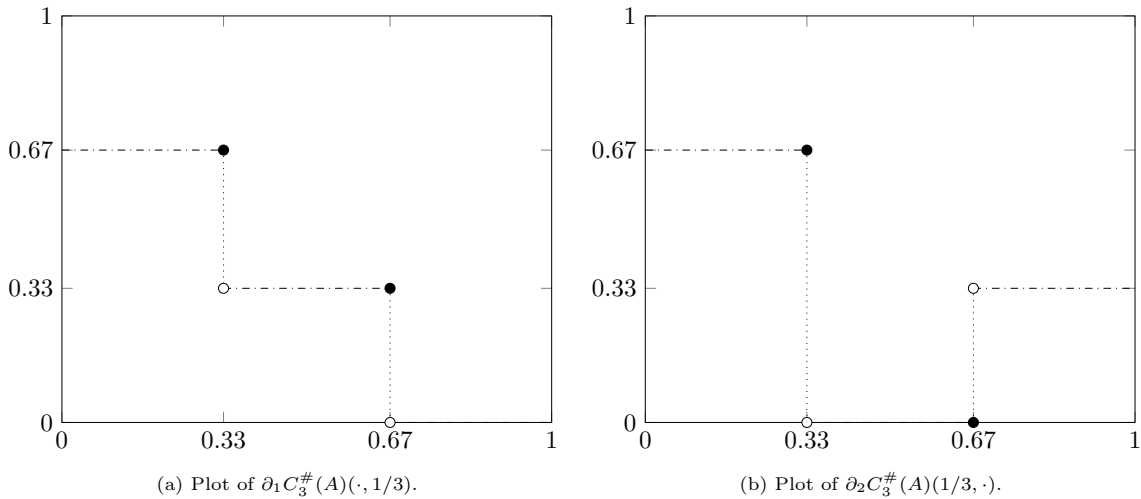
\end{example}

To simplify subsequent proofs, we will first present a direct connection between stochastically increasing and decreasing copulas, allowing us to transfer results obtained for stochastically increasing copulas to stochastically decreasing ones and vice versa.

\begin{lemma} \label{lemma:si_involution_sd}
The mapping $C \mapsto (C^- * C)$ is an involution between $\ir{\C}$ and $\dr{\C}$. 
\end{lemma}

\begin{proof}
The claim follows immediately from $\partial_1 (C^- * C)(u, v) = \partial_1 C(1-u, v)$ for all $u, v \in [0, 1]$ and any $2$-copula $C$.  
\end{proof}

The class of stochastically monotone copulas also provides additional structure to strengthen convergence properties since the monotonicity yields the equivalence of uniform convergence convergence, the pointwise convergence of the partial derivative and the weak conditional convergence introduced
in \cite{Kasper.2020}. 
The extension to weakly conditional convergent copulas was communicated to us by Wolfgang Trutschnig. 

\begin{proposition} \label{proposition:stochastic_monotone_convergence}
Let $C_n$ be a sequence of $2$-copulas, which are stochastically monotone in the first component. 
Then the following are equivalent: 
\begin{enumerate}
	\item $C_n$ converges uniformly towards $C$.\label{thm:convergence_pntw}
	\item $\partial_1 C_n(u, v)$ converges pointwise towards $\partial_1 C(u, v)$ for all $v$ in $[0, 1]$ and almost all $u$ in $[0, 1]$. \label{thm:convergence_derivative}
	\item $C_n$ converges weakly conditional towards $C$, i.e. the associated Markov kernels\footnote{see, e.g., \cite{Durante.2015} for a definition and a comprehensive overview on Markov kernels.} $K_{C_n}(u, \cdot)$ converge weakly towards $K_C(u, \cdot)$ for almost all $u \in [0, 1]$.  \label{thm:convergence_wcc}
\end{enumerate} 
\end{proposition}

\begin{proof}
Using Lemma~7 in \cite{Trutschnig.2011}, (\ref{thm:convergence_wcc}) implies (\ref{thm:convergence_pntw}).
Conversely, suppose $C_n$ converges uniformly towards $C$.
Due to $C_n(\cdot, v)$ and $C(\cdot, v)$ being concave (convex) for all $v \in [0, 1]$, Lemma~1 in \cite{Tsuji.1952} implies  
\begin{equation*}
	\lim\limits_{n \rightarrow \infty} \partial_1 C_n(u, v) = \partial_1 C(u, v)
\end{equation*}
for almost all $u \in [0, 1]$ and all $v \in [0, 1]$. 
This yields the assertion (\ref{thm:convergence_pntw}) to (\ref{thm:convergence_derivative}).
Lastly, suppose (\ref{thm:convergence_derivative}) holds, i.e. $\partial_1 C_n(u, v)$ converges pointwise towards $\partial_1 C(u, v)$ for all $v$ and almost all $u$.
Let $K_n$ and $K$ denote the Markov kernel associated with $C_n$ and $C$, respectively.
Then there exists for any $v \in [0, 1] \cap \mathbb{Q}$ a set $\Lambda_v$ such that $\lambda(\Lambda_v) = 1$ and $K_n(u, [0, v])	= \partial_1 C_n(u, v)$ holds for all $u \in \Lambda_v$. 
Thus,
\begin{equation*}
	K_n(u, [0, v])	\rightarrow K (u, [0, v])	\quad \text{ for all } u \in \Lambda := \bigcap_{v \in [0, 1] \cap \mathbb{Q}} \Lambda_v
\end{equation*}
and all $v \in [0, 1] \cap \mathbb{Q}$.
The pointwise convergence of the distribution functions on a dense set then yields the convergence in all continuity points.
Thus, $K_n(u, \cdot)$ converges weakly towards $K(u, \cdot)$ for almost all $u$ in $[0, 1]$ and the assertion follows.
\end{proof}

\begin{remark}
Proposition~\ref{proposition:stochastic_monotone_convergence} yields the equivalence of the uniform convergence and the $D_1$-convergence for stochastically monotone copulas, where the metric $D_1$ (see, \cite{Trutschnig.2011}) is defined by
\begin{equation*}
	D_1(C_1, C_2)	:= \lebesgue{\abs{\partial_1 C_1(u, v) - \partial_1 C_2(u, v)}}{[0, 1]^2}{\lambda(u, v)} ~.
\end{equation*}
Moreover, $\ir{\C}$ and $\dr{\C}$ are closed with respect to $d_\infty$ and $D_1$.
\end{remark}

\begin{remark}
The equivalence of (\ref{thm:convergence_pntw}) to (\ref{thm:convergence_derivative}) in Proposition~\ref{proposition:stochastic_monotone_convergence} also holds for completely dependent copulas and their corresponding Markov operators, so-called Markov embeddings, a proof of which can be found in Theorem~13.11 in \cite{Eisner.2015}.
\end{remark}

We will now give the main result of this section and  characterize the behaviour of stochastically monotone $2$-copulas and their corresponding Markov operators.
We say $f \in L^1([0,1])$ is monotone if there exists a monotone function $g: [0, 1] \rightarrow \R$ such that $f = g$ holds almost everywhere.

\begin{theorem}\label{thm:characterization}
Suppose $X$ and $Y$ are continuous random variables with copula $C$. 
Then the following are equivalent
\begin{enumerate}
	\item $Y$ is stochastically increasing (decreasing) in $X$. \label{thm:characterization_rv}
	\item $C$ is stochastically increasing (decreasing) in the first component.\label{thm:characterization_stoch}
	\item $C(u, v)$ is concave (convex) in $u$ for all $v \in [0, 1]$. \label{thm:characterization_con}
	\item $T_C$ is a monotonicity-preserving (monotonicity-reversing) Markov operator, i.e. $T$ maps decreasing integrable functions onto decreasing (increasing) functions.\label{thm:characterization_operator}
	\item $\mathbb{E} (f(Y) \mid X = x)$ is decreasing (increasing) for every decreasing function $f$ such that the expectation exists.\label{thm:characterization_expectation}
\end{enumerate}
\end{theorem}

\begin{proof}[Proof of Theorem \ref{thm:characterization}]
We give the proof for stochastically increasing random variables; the case of stochastically decreasing random variables is similar and left to the reader.
The equivalence of (\ref{thm:characterization_rv}), (\ref{thm:characterization_stoch}) and (\ref{thm:characterization_con}) is shown in \cite{Nelsen.2006}. 
Suppose (\ref{thm:characterization_operator}) holds, then $f = \indFunc{[0, v]}$ yields (\ref{thm:characterization_stoch}).
For the implication (\ref{thm:characterization_stoch}) to (\ref{thm:characterization_operator}), note that $T_C$ maps decreasing indicator functions onto decreasing functions due to
\begin{equation*}
	T_C \indFunc{[0, v]}(\cdot) = \partial_1 C(\cdot, v) 
\end{equation*}
being decreasing for all $v \in [0, 1]$. 
Using the approximation of monotone functions via monotone indicator functions and applying the monotone convergence theorem, (\ref{thm:characterization_operator}) follows as outlined in \cite{Mosler.1991}. 
Similarly, (\ref{thm:characterization_rv}) and (\ref{thm:characterization_expectation}) are equivalent using
\begin{equation*}
	\mathbb{E} (\indFunc{[0, y]}(Y) \mid X = x) = \mathbb{P} (Y \leq y \mid X = x) ~. \qedhere
\end{equation*}
\end{proof}
\section{The Markov product of stochastically monotone copulas} \label{section:markov_product}

The characterization given in Theorem~\ref{thm:characterization} guarantees that the composition of monotonicity-preserving Markov operators is again monotonicity-preserving.
Using the isomorphism between $(\Cd{2}, *)$ and Markov operators equipped with the composition, we establish the following closure property of $\ir{C}$ and $\dr{C}$ with respect to the Markov product.

\begin{corollary} \label{cor:stochastic_decreasing_markov}
Suppose $C_1, C_2 \in \Cd{2}$ are stochastically monotone in the first component. 
Then $C_1 * C_2$ is again stochastically monotone in the first component.
More precisely, $C_1 * C_2$ is 
\begin{enumerate}
	\item stochastically increasing if both $C_1$ and $C_2$ are either stochastically increasing or stochastically decreasing. \label{cor:stochastic_decreasing_markov_incr}
	\item stochastically decreasing if one is stochastically increasing and one is stochastically decreasing. \label{cor:stochastic_decreasing_markov_decr}
\end{enumerate}
\end{corollary}

\begin{proof}
If $C_1$ and $C_2$ are stochastically monotone in the first component, using Theorem~\ref{thm:characterization}, $T_{C_1}$ and $T_{C_2}$ map monotone functions onto monotone functions.
Their composition therefore also maps monotone functions onto monotone functions. 
Assertions~(\ref{cor:stochastic_decreasing_markov_incr}) and (\ref{cor:stochastic_decreasing_markov_decr}) follow immediately from a case-by-case analysis using Property~(\ref{thm:characterization_operator}) of Theorem~\ref{thm:characterization}.
\end{proof}

This closure property is only one aspect of the interplay between stochastically monotone $2$-copulas and the Markov product. 
The next result shows that stochastically increasing $2$-copulas maximize the Markov product.
A similar property was observed for tail dependence functions in Theorem~4.1 of \cite{Siburg.2021}.

\begin{theorem} \label{thm:characterization_mp_si}
Let $C$ be a $2$-copula.
$C$ is stochastically increasing in the first component if and only if
\begin{equation*} 
	(D * C)(u, v) \leq C(u, v) ~.
\end{equation*}
holds for all $2$-copulas $D$ and all $u, v  \in [0, 1]$.
On the other hand, $C$ is stochastically decreasing in the first component if and only if
\begin{equation*}
	C(u, v) \leq (D * C)(u, v) 
\end{equation*}
holds for all $2$-copulas $D$ and all $u, v  \in [0, 1]$. 
\end{theorem}%

Naturally, Theorem~\ref{thm:characterization_mp_si} also yields that the Markov operator $T_C$ is monotonicity-preserving if and only if
\begin{equation*}
	T_D \circ T_C f \leq T_C f
\end{equation*} 
for all decreasing functions $f \in L^1([0, 1])$ and all Markov operators $T_D$. 

\begin{proof}
We will only show the first equivalence, the second follows similarly and is left to the reader. 
Since
\begin{equation*}
	\cInt{\partial_2 D(u, t)}{0}{v}{t} = D(u, v) \leq C^+(u, v)	= \cInt{\indFunc{[0, u]}(t)}{0}{v}{t}
\end{equation*}
holds for all $u, v$ in $[0, 1]$, an application of Hardy's Lemma (see, Proposition~3.6 in \cite{Bennett.1988}) yields
\begin{equation*}
	(D * C)(u, v)	= \cInt{\partial_2 D(u, t) \partial_1 C(t, v)}{0}{1}{t}
					\leq \cInt{\indFunc{[0, u]}(t) \partial_1 C(t, v)}{0}{1}{t}
					= C(u, v) ~.  
\end{equation*}
Now, let us turn to the converse implication and assume $D * C \leq C$ holds for all $2$-copulas $D$.
Let $v \in (0, 1)$ be arbitrary and set $f(u) := \partial_1 C(u, v)$. 
Using Proposition~3 in \cite{Ryff.1970}, there exists a measure-preserving transformation $\sigma: [0, 1] \rightarrow [0, 1]$ and a decreasing function $g: [0, 1] \rightarrow [0, 1]$ such that
\begin{equation*}
	\partial_1 C(u, v)	= f(u) 
						= g(\sigma(u)) 
						= T_\sigma g (u) ~,
\end{equation*}
where $T_\sigma$ is a left-invertible Markov operator (commonly known as a Koopman operator).
Using Theorem~\ref{thm:cop_mo_correspondence}, $T_\sigma$ corresponds to a left-invertible $2$-copula $C_\sigma$. 
An application of the adjoint $T_\sigma'$ together with the left-inveribility of $T_\sigma$ yields
\begin{equation*}
	g(u)	= T_\sigma' \partial_1 C(\cdot, v) (u)
			= \partial_1 (C_\sigma^\top * C)(u, v) ~.
\end{equation*}
Setting $D := C_\sigma^\top * C$, we have $\partial_1 D(u, v) = g(u)$ almost everywhere.
Therefore $u \mapsto \partial_1 D(u, v)$ is decreasing and fulfils
\begin{equation*}
	D(u, v)	= (C_\sigma^\top * C)(u, v) \leq C(u, v) ~.
\end{equation*}
On the other hand, the Hardy-Littlewood-Inequality (see, (6.1) in \cite{Day.1972}) yields
\begin{equation*}
	C(u, v)	= \cInt{\partial_1 C(t, v)}{0}{u}{t}	
			\leq \cInt{\partial_1 D(t, v)}{0}{u}{t}
			= D(u, v)
			\leq C(u, v) ~.
\end{equation*}
Thus, $C(\cdot, v) = D(\cdot, v)$ and $u \mapsto \partial_1 C(u, v)$ must be decreasing. 
\end{proof}

\begin{remark}
Theorem~\ref{thm:characterization_mp_si} yields an alternative approach to derive the positive quadrant dependence of stochastically increasing copulas.
That is,
\begin{equation*}
	\Pi (u, v) 	= (\Pi * C)(u, v)	\leq C(u, v)  
\end{equation*}
holds for any stochastically increasing $2$-copula $C$.
\end{remark}

\begin{remark} \label{rem:sd_nqd}
Analogously to the previous remark, any stochastically decreasing $2$-copula $C$ is negative quadrant dependent due to 
\begin{equation*}
	C (u, v) 	\leq (\Pi * C)(u, v)	= \Pi(u, v)  ~.
\end{equation*}
\end{remark}

With a similar approach, it also characterizes stochastically increasing, completely dependent copulas.

\begin{remark}
A $2$-copula $C$ is called completely dependent, or left-invertible, if $C^\top * C = C^+$.
Due to Theorem~\ref{thm:characterization_mp_si}, any complete dependent and stochastically increasing copula $C$ fulfils
\begin{equation*}
	C^+ = C^\top * C \leq C^\top \leq C^+ 
\end{equation*}
so that $C = C^+$ holds. 
\end{remark}
\section{Idempotents of stochastically monotone copulas} \label{section:idempotents}

The rest of this article aims to characterize idempotent, stochastically monotone $2$-copulas and monotonicity-preserving conditional expectations.
While the idempotency of $C$ is appears to be a purely algebraic property, it translates to the fundamental stochastic property of $T_C$ being a conditional expectation on $L^\infty$.

It is well-known that any idempotent $2$-copula $C$ is necessarily symmetric, see for example \cite{Darsow.2010} or \cite{Trutschnig.2013b}.
Therefore, whenever it is stochastically monotone in one component, it is stochastically monotone in the same sense in the other component. 
Thus, it suffices to require $C$ to be stochastically monotone in either component and we will simply call $C$ stochastically monotone. 
With this in mind, let us state the main result of this section.

\begin{theorem} \label{thm:sm_copula}
Suppose $C$ is a $2$-copula.
$C$ is stochastically monotone and idempotent if and only if it is an ordinal sum of $\Pi$. 
\end{theorem}

We split the proof of Theorem~\ref{thm:sm_copula} into two parts.
We will begin with the result concerning stochastically decreasing copulas.

\begin{proposition} \label{thm:sd_copula}
The product copula $\Pi(u,v)=uv$ is the only idempotent $2$-copula which is stochastically decreasing.
\end{proposition}

\begin{proof}
Let $C$ be an arbitrary stochastically decreasing idempotent copula. 
Corollary~\ref{cor:stochastic_decreasing_markov} together with $C$ being idempotent yields that $C = C * C$ is stochastically increasing. 
Thus, $\partial_1 C(u, v) = c_v \in [0, 1]$ must hold for almost all $u \in [0, 1]$, which combined with the uniform margin property of copulas leads to
\begin{equation*}
	v = \cInt{\partial_1 C(u, v)}{0}{1}{u} = \cInt{c_v}{0}{1}{u} = c_v ~.
\end{equation*}
Integrating then gives the assertion $C(u, v) = uv = \Pi(u, v)$. 
\end{proof}

Proposition~\ref{thm:sd_copula} states that the only idempotent copula in the class of stochastically decreasing copulas is the product copula. It is natural to ask whether the same holds true inside the larger class of negative quadrant dependent copulas; see Remark~\ref{rem:sd_nqd}. Indeed, this is the case as the following proposition shows.

\begin{proposition} \label{thm:nqd_copula}
The product copula $\Pi(u,v)=uv$ is the only idempotent $2$-copula which is negative quadrant dependent.
\end{proposition}

\begin{proof}
Since every idempotent copula is symmetric (see Thm.~6.1 in \cite{Darsow.2010}) we can apply Prop.~17 from \cite{Siburg.2008} with an equality and, together with the assumption that $C(u,v) \leq \Pi(u,v)$, obtain that the so-called Sobolev norm $\|C\|$ satisfies $$ \|C\|^2 = 2 \cInt{(C*C)(u,u)}{0}{1}{u} = 2 \cInt{C(u,u)}{0}{1}{u} \leq 2 \cInt{\Pi(u,u)}{0}{1}{u} = \frac{2}{3} ~. $$ Now Thm.~18 in \cite{Siburg.2008} implies $\|C\|^2 = 2/3$ and, consequently, $C = \Pi$.
\end{proof}

Since Proposition~\ref{thm:sd_copula} already characterizes all idempotent, stochastically decreasing copulas, in the following, it remains to analyze the behaviour of stochastically increasing copulas. 
The next lemma provides a crucial technical property for the proof of Theorem~\ref{thm:sm_copula} by relating the partial derivative of a stochastically increasing copula $C$ to the rate of change from $C(v, v)$ to $v$ of $C$.

\begin{lemma} \label{lemma:tangent_property}
Let $C$ be an idempotent, stochastically increasing $2$-copula.
Then
\begin{equation*} 
	(v - C(v, v))\partial^-_2 C(u, v) = C(u, v) - C(u, C(v, v))
\end{equation*}
holds for all $u,v \in (0, 1)$, where $\partial^-_2 C$ denotes the left-hand derivative of $C$ with respect to the second component. 
\end{lemma}

\begin{proof}
Suppose $C$ is stochastically increasing, then the left-hand partial derivative $\partial^-_2 C(u, t)$ exists everywhere and is decreasing.
Furthermore, due to $C(t, v) \leq C^+(t, v) \leq t$, we have
\begin{align*}
	(C * C)(u, v)	&=  \cInt{\partial_2 C(u, t) \partial_1 C(t, v)}{0}{1}{t} 
					= \cInt{\partial^-_2 C(u, t) \partial_1 C(t, v)}{0}{1}{t} \\
				&\leq \cInt{\partial^-_2 C(u, C^+(t, v)) \partial_1 C(t, v)}{0}{1}{t} \\
				&\leq \cInt{\partial^-_2 C(u, C(t, v)) \partial_1 C(t, v)}{0}{v}{t} 
					+ \cInt{\partial^-_2 C(u, C^+(t, v)) \partial_1 C(t, v)}{v}{1}{t}  \\
				&\leq \cInt{\partial^-_2 C(u, C(t, v)) \partial_1 C(t, v)}{0}{1}{t} 
				= \cInt{\partial^-_2 C(u, z)}{C(0, v)=0}{C(1, v)=v}{z}
				= C(u, v) ~.
\end{align*}
The change of variables is possible due to the Riemann-integrability of $t \mapsto \partial_1 C(t, v)$ and $t \mapsto \partial_2 C(u, t)$.
Now, as $(C * C)(u, v) = C(u, v)$ holds, all inequalities are in fact equalities.
This yields
\begin{align*}
	C(u, v)		&= \cInt{\partial^-_2 C(u, C(t, v)) \partial_1 C(t, v)}{0}{v}{t} 
					+ \cInt{\partial^-_2 C(u, C^+(t, v)) \partial_1 C(t, v)}{v}{1}{t} \\
				&= \cInt{\partial^-_2 C(u, C(t, v)) \partial_1 C(t, v)}{0}{v}{t} 
					+ \cInt{\partial^-_2 C(u, v) \partial_1 C(t, v)}{v}{1}{t} \\					
				&= \cInt{\partial^-_2 C(u, z)}{C(0,v)=0}{C(v, v)}{z}
					+ \partial^-_2 C(u, v) \left(v - C(v, v) \right) \\
				&= C(u, C(v, v)) + \partial^-_2 C(u, v) \left(v - C(v, v) \right) ~. \qedhere
\end{align*}
\end{proof}

This property of the partial derivative is the cornerstone to prove the desired characterization of idempotent, stochastically increasing copulas.
The only difficulty remains in the term $(v - C(v, v)) \geq 0$.
Whenever the latter is strictly positive, the following lemma characterizes the corresponding copula completely. 
If $(v - C(v, v)) = 0$, we will need to consider the behaviour of $C$ more closely in Theorem~\ref{thm:si_copula}.

\begin{lemma} \label{lemma:idempotent_stochastic_monotone}
Suppose $C$ is a $2$-copula with $C(v, v) < v$ for all $v \in (0, 1)$.
Then $C$ is stochastically monotone and idempotent if and only if $C(u, v) = uv = \Pi(u, v)$. 
\end{lemma}

\begin{proof}
The assertion for stochastically decreasing copulas follows from Proposition~\ref{thm:sd_copula}.
Thus, applying Lemma~\ref{lemma:tangent_property} in combination with $C$ being stochastically increasing, we obtain for arbitrary $u, v$ in $[0, 1]$
\begin{align*}
	\partial^-_2 C(u, v) 	&= \frac{C(u, v) - C(u, C(v, v))}{v - C(v, v)} 
							= \frac{1}{v - C(v, v)} \cInt{\partial^-_2 C(u, t)}{C(v, v)}{v}{t} \\
							&\geq \frac{1}{v - C(v, v)} \cInt{\partial^-_2 C(u, v)}{C(v, v)}{v}{t}
							= \partial^-_2  C(u, v) ~.
\end{align*}
Therefore, the inequality above is actually an equality, and the partial derivative, not only the left-hand derivative, fulfils $\partial_2 C(u, t) =  c_u \in [0, 1]$ almost everywhere on $(C(v, v), v)$. 
Since $C(v,v) < v$ holds for all $v \in (0, 1)$, we obtain a (non-disjoint) covering $\curly{(C(v, v), v)}_{v \in (0, 1)}$ of $(0, 1)$ with intervals having nonempty interior. 
Consequently, we must have $\partial_2 C(u, t) = c_u$ for almost all $t \in (0,1)$. 
Hence,
\begin{equation*}
	u = C(u, 1)	= \int\limits_0^1 \partial_2 C(u, t) \mathrm{d} t 
				= \int\limits_{0}^1 c_u \mathrm{d} t 
				= c_u 
\end{equation*}
which implies $C(u, v) = uv = \Pi(u, v)$. 
\end{proof}

\begin{theorem} \label{thm:si_copula}
Ordinal sums of $\Pi$ are the only idempotent $2$-copulas which are stochastically increasing. 
\end{theorem}

\begin{proof}
The proof treats three distinct cases, depending on the behaviour along the diagonal.
Suppose $C$ is stochastically increasing and idempotent. 
If $C(v, v) < v$ for all $v \in (0, 1)$, then $C = \Pi = \ordSum{(0, 1)}{\Pi}$ due to Lemma~\ref{lemma:idempotent_stochastic_monotone}.
If on the other hand $C(v, v) = v$ holds for all $v \in (0, 1)$, then $C = C^+ = \ordSum{(a_k, b_k)}{\Pi}_{k \in \emptyset}$. 
Lastly, if $C(v, v) = v$ holds for some $v \in (0, 1)$ and $C \neq C^+$, Corollary~3.2 and 3.3 from \cite{Mesiar.2010} yield that $C$ is the ordinal sum of ordinally irreducible $2$-copulas
\begin{equation*}
	C = \rbraces{\ordSum{(a_k, b_k)}{C_k}}_{k \in \mathcal{I}} ~.
\end{equation*}
Due to Theorem~3.2.1 in \cite{Nelsen.2006}, ordinally irreducible copulas $C_k$ fulfil $C_k(v, v) < v$ for all $v \in (0, 1)$.
Theorem~3.1 from \cite{Albanese.2016} then states that $C$ is idempotent if and only if every $C_k$ is idempotent. 
Moreover, the ordinal sum $C$ is stochastically increasing if and only if every $C_k$ is stochastically increasing.
Thus, every $C_k$ is idempotent, stochastically increasing and fulfils $C_k(v, v) < v$ on $(0, 1)$.
Lemma~\ref{lemma:idempotent_stochastic_monotone} then implies $C_k = \Pi$ which yields 
\begin{equation*}
	C = \rbraces{\ordSum{(a_k, b_k)}{\Pi}}_{k \in \mathcal{I}} ~. \qedhere
\end{equation*} 
\end{proof}

\begin{example}
Apart from the extreme cases $\Pi$ and $C^+$, ordinal sums of $\Pi$ can take various forms.
Three different configurations, namely 
\begin{equation*}
	\rbraces{\ordSum{\rbraces{0, \frac{1}{3}}}{\Pi}, \ordSum{\rbraces{\frac{5}{6}, 1}}{\Pi}}, \ordSum{\rbraces{\frac{1}{3}, 1}}{\Pi} \text{ and } \rbraces{\ordSum{\rbraces{\frac{k}{6}, \frac{k+1}{6}}}{\Pi}}_{k \in \curly{0, \ldots, 5}} ~, 
\end{equation*}
are depicted in Figure~\ref{fig:idempotent_stochastic_decreasing}.
\end{example}

\begin{figure}
\begin{subfigure}[b]{0.32\textwidth}
	\centering
    \resizebox{\linewidth}{!}{\begin{tikzpicture}
	\fill[gray!20!white] (0,0) rectangle (2,2);
	\fill[gray!20!white] (5,5) rectangle (6,6);
	\draw (0,0) node[anchor=north east] {$(0,0)$} -- (6,0) node[anchor=north west] {$(1,0)$} -- (6,6)  node[anchor=south west] {$(1,1)$} -- (0,6)  node[anchor=south east] {$(0,1)$} -- (0,0);
	\draw[dotted] (0,0) -- (2,0) -- (2,2) -- (0,2) -- (0,0);
	\draw[dotted] (2,2) -- (5,2) -- (5,5) -- (2,5) -- (2,2);
	\draw[dotted] (5,5) -- (6,5) -- (6,6) -- (5,6) -- (5,5);
	\draw (2,2) -- (5,5);
\end{tikzpicture}}
    \caption{$\rbraces{\ordSum{(0, \frac{1}{3})}{\Pi}, \ordSum{(\frac{5}{6}, 1)}{\Pi}}\vphantom{\rbraces{\ordSum{(\frac{k}{6}, \frac{k+1}{6})}{\Pi}}_{k \in \curly{0, \ldots, 5}}}$.}
    \label{figSubA:idempotent_stochastic_decreasing0}
\end{subfigure}
\begin{subfigure}[b]{0.32\textwidth}
	\centering
    \resizebox{\linewidth}{!}{\begin{tikzpicture}
	\fill[gray!20!white] (2,2) rectangle (6,6);
	\draw (0,0) node[anchor=north east] {$(0,0)$} -- (6,0) node[anchor=north west] {$(1,0)$} -- (6,6)  node[anchor=south west] {$(1,1)$} -- (0,6)  node[anchor=south east] {$(0,1)$} -- (0,0);
	\draw[dotted] (0,0) -- (2,0) -- (2,2) -- (0,2) -- (0,0);
	\draw[dotted] (2,2) -- (6,2) -- (6,6) -- (2,6) -- (2,2);
	\draw (0,0) -- (2,2);
\end{tikzpicture}}
    \caption{$\ordSum{(\frac{1}{3}, 1)}{\Pi}\vphantom{\ordSum{a}{b}^1_{k \in \curly{1}}}\vphantom{\rbraces{\ordSum{(\frac{k}{6}, \frac{k+1}{6})}{\Pi}}_{k \in \curly{0, \ldots, 5}}}$.}
    \label{figSubA:idempotent_stochastic_decreasingv0}
\end{subfigure}
\begin{subfigure}[b]{0.32\textwidth}
	\centering
    \resizebox{\linewidth}{!}{\begin{tikzpicture}
	\fill[gray!20!white] (0,0) rectangle (1,1);
	\fill[gray!20!white] (1,1) rectangle (2,2);
	\fill[gray!20!white] (2,2) rectangle (3,3);
	\fill[gray!20!white] (3,3) rectangle (4,4);
	\fill[gray!20!white] (4,4) rectangle (5,5);
	\fill[gray!20!white] (5,5) rectangle (6,6);
	\draw (0,0) node[anchor=north east] {$(0,0)$} -- (6,0) node[anchor=north west] {$(1,0)$} -- (6,6)  node[anchor=south west] {$(1,1)$} -- (0,6)  node[anchor=south east] {$(0,1)$} -- (0,0);
	\draw[dotted] (0,0) -- (1,0) -- (1,1) -- (0,1) -- (0,0);
	\draw[dotted] (1,1) -- (2,1) -- (2,2) -- (1,2) -- (1,1);
	\draw[dotted] (2,2) -- (3,2) -- (3,3) -- (2,3) -- (2,2);
	\draw[dotted] (3,3) -- (4,3) -- (4,4) -- (3,4) -- (3,3);
	\draw[dotted] (4,4) -- (5,4) -- (5,5) -- (4,5) -- (4,4);
	\draw[dotted] (5,5) -- (6,5) -- (6,6) -- (5,6) -- (5,5);
\end{tikzpicture}}
    \caption{$\rbraces{\ordSum{(\frac{k}{6}, \frac{k+1}{6})}{\Pi}}_{k \in \curly{0, \ldots, 5}}$.}
    \label{figSubA:idempotent_stochastic_decreasing1}
\end{subfigure}
\caption{Stochastically monotone idempotent ordinal sums of $\Pi$ with respect to different families of disjoint intervals.}
\label{fig:idempotent_stochastic_decreasing}
\end{figure}
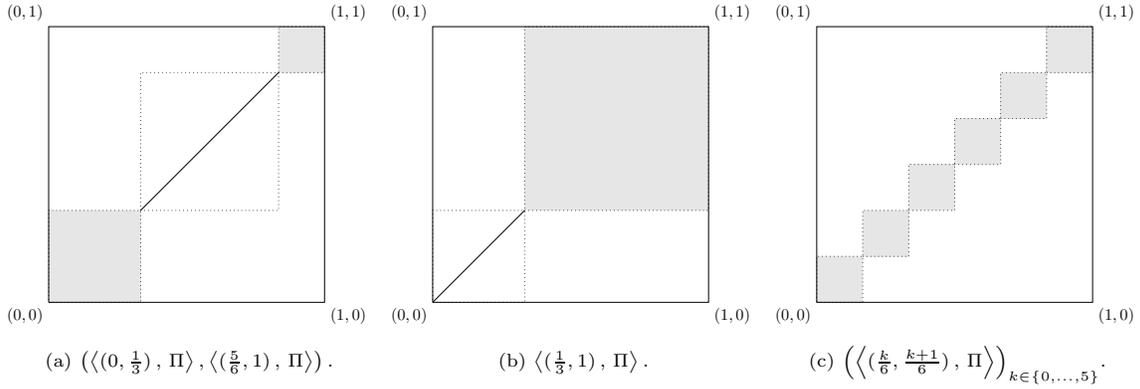

\begin{proposition}
Suppose $C$ is stochastically increasing in the first component. 
Then there exists a family of intervals $(a_k, b_k)_{k \in \mathcal{I}}$ such that
\begin{equation*}
	C^{*n}(u, v) \rightarrow \rbraces{\ordSum{(a_k, b_k)}{\Pi}}_{k \in \mathcal{I}}(u, v)
\end{equation*}
converges pointwise. 
\end{proposition}

A similar behaviour was established for Cesáro averages of iterates of quasi-constrictive Markov operators in \cite{Trutschnig.2015}.

\begin{remark}
Since the $2$-copulas $C^{*n}$ are stochastically increasing in the first component (see Corollary~\ref{cor:stochastic_decreasing_markov}), an application of Proposition~\ref{proposition:stochastic_monotone_convergence} yields that the pointwise convergence of
\begin{equation*}
	C^{*n}(u, v) \rightarrow \rbraces{\ordSum{(a_k, b_k)}{\Pi}}_{k \in \mathcal{I}}(u, v)
\end{equation*}
is equivalent to the pointwise convergence of the partial derivatives and the weakly conditional convergence. 
\end{remark}

\begin{proof}
Due to Corollary~\ref{cor:stochastic_decreasing_markov}, $C^{*n}$ is stochastically increasing for all $n \in \mathbb{N}$.
By Theorem~\ref{thm:characterization_mp_si},
\begin{equation*}
	0 \leq C^{*n}(u, v)	= (C * C^{*(n-1)})(u, v) \leq C^{*(n-1)}(u, v)
\end{equation*}
follows for all $u, v \in [0, 1]$. 
Thus, $C^{*n}$ is a decreasing sequence of copulas and as such, converges pointwise against some $C^* \in \Cd{2}$.
The pointwise limit of the concave functions $u \mapsto C^{*n}(u, v)$ is again concave, therefore $C^*$ is stochastically increasing in the first component.
Furthermore, due to the Markov product being continuous with respect to the pointwise convergence in one component, we have that
\begin{equation*}
	C * C^*	= \lim\limits_{n \rightarrow \infty} C * C^{*n} = C^* ~.
\end{equation*} 
holds.
An inductive arguments now yields
\begin{equation*}
	C^* * C^* = \lim\limits_{n \rightarrow \infty} C^{*n} * C^* = \lim\limits_{n \rightarrow \infty} C^* = C^* ~.
\end{equation*}
Thus $C^*$ is idempotent. 
An application of Theorem~\ref{thm:si_copula} then guarantees the existence of a family of intervals $(a_k, b_k)_{k \in \mathcal{I}}$ such that $C^* =\rbraces{\ordSum{(a_k, b_k)}{\Pi}}_{k \in \mathcal{I}}$.   
\end{proof}

Finally, we will translate the results of this section into the language of Markov operators and conditional expectations.
Following Proposition~\ref{prop:idempotent_CE}, a $2$-copula $C$ is idempotent if and only if $T_C$ is a conditional expectation restricted to $L^\infty([0, 1], \mathcal{B}([0, 1]), \lambda)$.
Thus, it follows immediately that a conditional expectation is monotonicity-preserving if and only if it is pointwise either an average operator or the identity.

\begin{theorem} \label{thm:conditional_expectation_char}
Suppose $T$ is a conditional expectation on $L^\infty([0, 1], \mathcal{B}([0, 1]), \lambda)$.
Then $T$ preserves or reverses the monotonicity if and only if there exists a countable family of disjoint intervals $\rbraces{(a_k, b_k)}_{k \in \mathcal{I}}$ in $(0, 1)$ with $P := \cup_{k \in \mathcal{I}} (a_k, b_k)$ such that
\begin{equation*}
	T_C f (u)	= 	\sum\limits_{k = 0}^{\abs{\mathcal{I}}}  \indFunc{(a_k, b_k)}(u) \frac{1}{b_k - a_k} \cInt{f(t)}{a_k}{b_k}{t} + \indFunc{P^C}(u) f(u) ~.
\end{equation*}
\end{theorem}

\section*{Acknowledgements}

We thank Wolfgang Trutschnig for his constructive comments on an earlier version of this manuscript.  
We thank an anonymous referee for their constructive comments, which lead to Proposition~\ref{thm:nqd_copula} and helped us to improve the article.
The second author wishes to thank Wolfgang Trutschnig and his group at the University of Salzburg for their hospitality and stimulating discussions. 
The second author gratefully acknowledges financial support from the German Academic Scholarship Foundation.

\bibliography{reference}

\end{document}